\documentclass[11pt]{article}
\usepackage{amsfonts}
\usepackage{amsmath,amssymb,amsthm,latexsym}
\usepackage{amscd}
\usepackage[all]{xy}
\usepackage{hyperref}
\usepackage{mathrsfs}
\usepackage{fancyhdr}
\usepackage{indentfirst}
\usepackage{subcaption}
\usepackage{titlesec}
\usepackage{color}
\usepackage{tabularx}
\usepackage{multirow}
\usepackage{listings}
\usepackage{cite}
\usepackage{float}
\usepackage[all]{xy}
\usepackage{graphicx}
\usepackage{lineno}
\usepackage{marvosym}
\usepackage{geometry}
\usepackage{authblk}
\newtheorem{theorem}{Theorem}[section]
\newtheorem{proposition}[theorem]{Proposition}
\newtheorem{corollary}[theorem]{Corollary}
\newtheorem{lemma}[theorem]{Lemma}
\newtheorem{problem}{Problem}
\newtheorem{conjecture}{Conjecture}

\theoremstyle{definition}

\newtheorem{rem}[theorem]{Remark}

\newcommand{\R}{\mathbb R}
\newcommand{\C}{\mathbb C}
\newcommand{\CP}{\mathbb{CP}}
\newcommand{\RP}{\mathbb{RP}}
\newcommand{\Sym}{\operatorname{Sym}}
\newcommand{\tr}{\operatorname{tr}}
\newcommand{\Id}{\operatorname{Id}}
\newcommand{\Area}{\operatorname{Area}}
\newcommand{\Vol}{\operatorname{Vol}}

\newcommand{\CL}{\mathrm{CL}}
\newcommand{\dd}{\mathrm d}
\begin{document}
\title{A proof of the Willmore-type conjecture in $\mathbb{C}P^2$
}
\author{Peng Wang\textsuperscript{1},\quad Zhenxiao Xie 
\textsuperscript{\Letter~2},\quad Chen Zhao\textsuperscript{3}}

\date{}
 \maketitle
\footnotetext[1]{School of Mathematics and Statistics, Key Laboratory of Analytical Mathematics and Applications (Ministry of Education), FJKLAMA, Fujian Normal University, 350117 Fuzhou, P.R. China. Email: {pengwang@fjnu.edu.cn}}
\footnotetext[2]{School of Mathematical Sciences, Beihang University,  Beijing 100191, P. R. China. %
Email: {xiezhenxiao@buaa.edu.cn}}
\footnotetext[3] {
School of Mathematical Sciences, Beihang University,  Beijing 100191, P. R. China. 
Email: {sc2zhaochen@hotmail.com}}

\vspace{-0.25 cm}
\begin{abstract}
    In 2002, Montiel and Urbano conjectured that the Clifford torus in \(\mathbb{C}P^2\) minimizes the Willmore-type functional
\[
\mathcal{W}^-=\int_{T^2}(2+|H|^2)\,dA,
\]
either among all tori or among all Lagrangian tori. In this paper, we confirm this conjecture in the Lagrangian setting and disprove it in the general setting. We establish that every oriented closed Lagrangian surface of genus \(g\geq 1\) in \(\mathbb{C}P^2\) has \(\mathcal{W}^-\)-energy no less than that of the Clifford torus. Moreover, we construct non-Lagrangian deformations of the Clifford torus along which \(\mathcal{W}^-\) strictly decreases.
\end{abstract}
\vspace{0.5em}

\noindent{\bf MSC(2020):\hspace{2mm} 53C42, 53D12, 43Q10}
\section{Introduction}
For an immersed closed surface $F:\Sigma\rightarrow M^n$ in a Riemannian manifold, it is well known that 
$$\mathcal{W}(F) = \frac{1}{2} \int_{\Sigma} |\mathring{\mathrm{II}}|^2 \,\dd A + 2\pi\chi(\Sigma)$$
is a fundamental global conformal invariant; that is, it is preserved under conformal changes of the ambient metric.  Here $\mathring{\mathrm{II}}$ denotes the trace-free part of the second fundamental form, and $\chi(\Sigma)$ is the Euler characteristic of $\Sigma$. This functional, referred to as the conformally invariant Willmore functional \cite{Hu-Li, Michelat-Mondino, Mondino-Riviere}, will henceforth be called simply the {\em Willmore functional}.

When the ambient space is real space forms, since the seminal works of Blaschke and Thomsen \cite{Blaschke}, the Willmore functional has been the subject of extensive investigation (see \cite{Bryant, L-Y, Kuwert, KS, Minicozzi, M-R, Marques-Neves2, Riviere} and references therein). Most of this research was dedicated to resolving the famous Willmore conjecture, first proposed by Willmore \cite{Willmore}, which posited that the Clifford torus in $\mathbb{S}^3$ minimizes the Willmore functional among all tori in $\mathbb{S}^3$. This conjecture was ultimately proven by Marques and Neves \cite{Marques-Neves}. 

When the ambient space is a generic Riemannian manifold, the Willmore functional has also been investigated from various perspectives (see \cite{Pedit-Willmore,MU,Hu-Li,Ikoma-Malchiodi-Mondino,Michelat-Mondino,Mondino-Riviere} and the references therein). In particular, it exhibits notable connections with other fundamental quantities, among them the renormalized area functional within the AdS/CFT correspondence \cite{Alexakis-Mazzeo,Graham-Witten}. Due to the presence of ambient curvature terms, the minimizing problem for the Willmore functional is considerably more challenging. Nevertheless, in $4$-dimensional oriented Riemannian manifolds, the geometry is rich enough that the Willmore functional of closed orientable surfaces exhibits a number of interesting properties.
In such ambient spaces, motivated by twistor geometry, Montiel and Urbano showed in \cite{MU} that $\mathcal{W}$ decomposes as the average of two other conformally invariant functionals $\mathcal{W}^+$ and $\mathcal{W}^-$.

In that paper, Montiel and Urbano focused on the complex projective plane $\mathbb{C}P^2$, endowed with its Fubini-Study metric $g_{\mathrm{FS}}$ of holomorphic sectional curvature $4$, complex structure $J$, and Kähler form $\omega$. For an immersed closed orientable surface $F:\Sigma\to\mathbb{C}P^2$, they proved that

\[
\mathcal{W}(F)=\int_{\Sigma}(1+3C^2+|H|^2)\,\dd A,
\]
\[
\mathcal{W}^+(F)=\int_{\Sigma}(6C^2+|H|^2)\,\dd A,\qquad
\mathcal{W}^-(F)=\int_{\Sigma}(2+|H|^2)\,\dd A=2\mathcal{A}(F)+\int_{\Sigma}|H|^2\,\dd A,
\]
where $H=\frac{1}{2}\operatorname{tr}\mathrm{II}$ denotes the mean curvature vector, $C$ is the Kähler function of $F$, defined by $F^*\omega=C\,dA$, and $\mathcal{A}$ is the area functional.
They showed that  $\mathcal{W}^-$ behaves as the Willmore functional of closed surfaces in the round sphere $\mathbb{S}^n$. For example,  minimal surfaces in $\mathbb{C}P^2$ are critical points of $\mathcal{W}^-$. They also obtained the following lower bound for $\mathcal{W}^-$, analogous to the Li–Yau bound  \cite{L-Y} for the Willmore functional in $\mathbb{S}^n$. 
\begin{theorem}[Montiel-Urbano \cite{MU}] \label{thm-MU}Let $F:\Sigma \rightarrow \mathbb{C}P^2$ be an immersed closed surface. Then 
    $$\mathcal{W}^-(F)\geq 2\pi\,\mu(F),$$ 
    where $\mu(F)$ denotes the maximum multiplicity of $F$.  
If $F$ is further assumed to be  Lagrangian, then 
$$\mathcal{W}^-(F)\geq 4\pi\,\mu(F).$$
\end{theorem}
Inspired by the classical Willmore conjecture in $\mathbb{S}^3$, Montiel and Urbano proposed the following Willmore-type conjecture for the Clifford torus 
$$T_{\CL} \triangleq \left\{ [z_1, z_2, z_3] \in \mathbb{C}P^2 \;\middle|\; |z_1| = |z_2| = |z_3| \right\},
$$ 
which is minimal and Lagrangian. For this torus, 
$$ \mathcal{A}(T_{\CL})=\frac{4\pi^2}{3\sqrt3},\qquad
 \mathcal{W}^-(T_{\CL})=2\mathcal{A}(T_{\CL}),\qquad \mathcal{W}(T_{\CL})=\mathcal{A}(T_{\CL}).$$
\vskip 0.2cm
\noindent
\textbf{Conjecture} (Montiel-Urbano \cite{MU}).  \textit{The Clifford torus $T_{\CL}$ achieves the minimum of the functional $\mathcal{W}^-$ either amongst all tori in $\mathbb{C}P^2$, or amongst all Lagrangian tori in $\mathbb{C}P^2$.}
\vskip 0.2cm

In the survey paper \cite{Marques-Neves2}, Marques and Neves mentioned the above 
Willmore-type conjecture in $\mathbb{C}P^2$ and pointed out that if answered positively, the conjecture would help in finding the nontrivial Special Lagrangian cone in $\mathbb{C}^3$ with least possible density. 
This conjecture has been verified on certain families of Hamiltonian-minimal Lagrangian tori, see \cite{Ma-Mironov-Zuo} and \cite{Kazhymurat}. Recently, C. P. Wang and the second author reinterpreted $\mathcal{W}^+$ and $\mathcal{W}^-$ from the perspective of conformal geometry \cite{WangXie}. Moreover, in that paper, they proved that the Clifford torus $T_{\CL}$ is strictly Willmore-stable, and asserted that it is also stable for the functional $\mathcal{W}^-$ in Remark~6.4. 


In this paper, we completely solve Montiel-Urbano's conjecture in two different directions. For the Lagrangian setting, we affirm it by applying the Chern-Lashof    inequality to the total absolute curvature of the Hopf bundle of Lagrangian surfaces. For the general setting, we disprove the conjecture by constructing counterexamples, obtained by deforming the Clifford torus along certain non-Killing Jacobi fields.

To be concrete, in the Lagrangian setting, instead of studying Lagrangian surfaces in $\mathbb{C}P^2$ directly, we consider the $3$-dimensional submanifolds in $\mathbb{S}^5\subset\mathbb{R}^6$ constructed from their Hopf bundles. In conformal geometry of submanifolds, such a construction has been used to construct Wintgen ideal submanifolds from complex curves in $\mathbb{C}P^2$; see \cite{DDVV,LMWX}. For the $3$-dimensional submanifolds constructed here, the squared length of the trace-free part of the second fundamental form is closely related to the integrand of $\mathcal{W}^-$. By further establishing a determinant integral inequality (see Theorem~\ref{thm:matrix}), we also relate this squared length to the total absolute curvature of the $3$-dimensional submanifold. The Chern-Lashof inequality then allows us to obtain the following proof of Montiel-Urbano's conjecture in the Lagrangian setting. 
\newtheorem*{newth}{Main Theorem}
\begin{newth}
{\em Let $F:\Sigma \rightarrow \mathbb{C}P^2$ be a closed orientable Lagrangian surfaces of genus $g\geq 1$. Then, 
$$\mathcal{W}^-(F)\geq \frac{8\pi^2}{3\sqrt{3}}=\mathcal W^-(T_{\CL}),\qquad\quad\mathcal{W}(F)\geq \frac{4\pi^2}{3\sqrt{3}}=\mathcal W(T_{\CL})$$
with equality attained if and only if $F$ is congruent to the Clifford torus in $\mathbb{C}P^2$. }
\end{newth} 

For the  general setting, by analyzing the second variation of $\mathcal{W}^-$ along all smooth variations, we show that the space of non-Killing Jacobi fields of the Clifford torus has dimension $6$. Using these Jacobi fields, we construct a $3$-parameter family of deformations of the Clifford torus that decrease the functional $\mathcal{W}^-$. This disproves Montiel-Urbano's conjecture in the non-Lagrangian setting. \vspace{3mm}

This paper is organized as below. In Section~\ref{sec:hopf}, we study the $3$-dimensional submanifolds in $\mathbb{S}^5\subset\mathbb{R}^6$ constructed from the Hopf bundles of Lagrangian surfaces in $\mathbb{C}P^2$. In Section~\ref{sec-proof}, 
we first interpret the required determinant integral as the area with multiplicity of a quadratic map, thereby reducing the problem to a trigonometric integral inequality. Based on this inequality, we finish the proof of the Main Theorem. In Section~\ref{sec:counterexamples}, we construct counterexamples to disprove Montiel-Urbano's conjecture in the non-Lagrangian setting. Finally, we proposed two open problems in Section~\ref{sec-5}.

\section{The Hopf bundle over Lagrangian surfaces in $\mathbb{C}P^2$}\label{sec:hopf}
Let $\pi:\mathbb{S}^5\to\CP^2$ be the Hopf fibration. For an immersion $F:\Sigma\to\CP^2$, the pull-back of the Hopf fibration 
\[
 \widehat\Sigma=F^*\mathbb{S}^5
   =\{(p,z)\in \Sigma \times \mathbb{S}^5:F(p)=\pi(z)\}
\]
yields a 3-dimensional manifold, and the map $\widehat F (p,z)=z:\widehat\Sigma\to \mathbb{S}^5\subset\R^6$ defines an immersion. 


Now, we assume that $F$ is Lagrangian. At a base point choose an oriented orthonormal frame $e_1,e_2$ such
that $H=hJe_1$, where $h=|H|\geq0$. When $H=0$, any such frame can
be used. Consider the  cubic tensor 
\[
 C_{ijk}=\langle \mathrm{II}(e_i,e_j),Je_k\rangle,
\]
which is totally symmetric under the Lagrangian condition. 
It follows that 
$$C_{111}+C_{122}=C_{111}+C_{221}=2h,~~~~~~C_{112}+C_{222}=0,$$
and we can parameterize the curvature ellipse of $F$ as follows, 
$$\mathrm{II}(\cos\theta \,e_1+\sin\theta\, e_2, \cos\theta \,e_1+\sin\theta\, e_2)=h Je_1+(\cos 2\theta, \sin 2\theta)\begin{pmatrix}
    C_{111}-h&C_{112}\\
    C_{112}&C_{122}
\end{pmatrix}\begin{pmatrix}
    Je_1\\
    Je_2
\end{pmatrix}.$$
Inspired by this, we write 
\begin{equation}\label{eq:cubic}
 C_{111}=a+\tfrac32h,\quad C_{112}=b,\quad
 C_{122}=-a+\tfrac12h,\quad C_{222}=-b,
\end{equation}
where $a,b$ are two locally defined functions on $\Sigma$. They are related to the locally defined conformal invariants $\phi$ and $\psi$ in \cite{WangXie} as follows, 
$$\phi=-a+i b,~~~\psi=-\frac{h}{2}.$$ 
We point out that for Lagrangian surfaces, the Willmore-type functional $\mathcal{W}^-$ can be reformulated, in terms of the local conformal invariant $\phi$, as
\begin{equation}\label{eq-W-}
\mathcal{W}^-(F)=4\int_{\Sigma}|\phi|^2+4\pi \,\chi(\Sigma), 
\end{equation}
see Remark~2.11 in \cite{WangXie}. 

Note that the ambient sectional curvature of a Lagrangian plane is $1$. Thus, by the Gauss equation (for example, see (22) in \cite{WangXie}), we obtain  
\begin{equation}\label{eq:gauss}
 K=1+\tfrac12h^2-2|\phi|^2.
\end{equation}

Next, we use the O'Neill formula for the Hopf fibration to compute the shape operator of $\widehat F$. It follows that for horizontal lifts $X,Y$,
\begin{equation}\label{eq-Oneil}
\nabla^{\mathbb{S}^5}_X Y
=\widetilde{\nabla^{\CP^2}_X Y}
-\langle JX,Y\rangle\,Jz,~~~~~~\nabla^{\mathbb{S}^5}_{Jz}X=JX.
\end{equation}
where $J$ is the complex structure on $\mathbb{C}^3\cong \mathbb{R}^6$, and $Jz$ denotes the vertical unit vector. 
Let $E_1,E_2$ denote the horizontal lifts of $e_1,e_2$, and set $E_0\triangleq Jz.$
Then $\{E_0, E_1,E_2\}$ forms an orthonormal tangent frame on $\widehat\Sigma$, and 
\[
N_1\triangleq JE_1,\qquad N_2\triangleq JE_2.
\]
forms an orthonormal normal frame along $\widehat F$ in $\mathbb{S}^5$.   
A straightforward calculation using \eqref{eq-Oneil} yields the corresponding shape operators, 
\begin{equation}\label{eq:shape-matrices}
 A_1=\begin{pmatrix}
 0&1&0\\1&a+3h/2&b\\0&b&-a+h/2
 \end{pmatrix},\qquad
 A_2=\begin{pmatrix}
 0&0&1\\0&b&-a+h/2\\1&-a+h/2&-b
 \end{pmatrix}.
\end{equation}

It is easy to verify that 
\begin{equation}\label{eq-norm}
|\mathring A_1|^2+|\mathring A_2|^2=4+4|\phi|^2+\frac5{3}h^2=4K+12|\phi|^2-\frac1{3}h^2.
\end{equation} 
This enables the application of the geometric inequality of the 3-dimensional submanifold $\widehat{F}$ to estimate the Willmore-type functional $\mathcal{W}^-$. A remarkable inequality concerning the shape operator of submanifolds is the Chern-Lashof inequality of the total absolute curvature, established in \cite{CL1,CL2}. 


For a closed immersed 3-dimensional submanifold $M^3$ in $\R^6$, the total absolute curvature $\tau(M)$ is defined as 
\begin{equation}\label{eq-tauM}\tau(M^3)\triangleq \frac{1}{\Vol(\mathbb{S}^5)}
   \int_{M^3}\left(\int_{\mathbb{S}^2\subset N_pM^3}|\det A_\nu|\,\dd S \right)\,\dd V,
   \end{equation}
where $\dd S$ denotes the standard area form on $\mathbb{S}^2$, viewed as the unit $2$-sphere in the normal space $N_p M^3$ of $M^3$ at $p$, $A_\nu$ is the shape operator with respect to the normal vector $\nu\in\mathbb{S}^2$, and $\dd V$ is the volume form on $M^3$. 
The Chern-Lashof inequality states that 
\begin{equation}\label{eq:chern-lashof}
 \tau(M^3)
 \ \geq\ \sum_j b_j(M^3;\R),
\end{equation}
where $b_j(M^3; \mathbb{R})$ denotes the $j$-th Betti number of $M^3$.

Next, let us return to the Hopf bundle $\widehat{\Sigma}$ over the Lagrangian surface $F:\Sigma \to \mathbb{C}P^2$. To determine the topology of $\widehat{\Sigma}$, we need to compute the Euler class of this circle bundle, which belongs to $H^2(\Sigma, \mathbb{Z})$. Let us now assume that $\Sigma = T^2$. 
Consider the pull-back of the Euler class of the Hopf bundle $\pi : S^5 \to \mathbb{C}P^2$. For this principal $S^1$-bundle, the Euler class naturally identifies with the first Chern class of the associated line bundle, which serves as the standard generator of $H^2(\mathbb{C}P^2, \mathbb{Z})$. In de Rham cohomology, this Chern class is represented by the Kähler form $\omega$ (up to a normalization constant). Consequently, the Lagrangian condition on $F$ implies that the pull-back yields $F^*\omega = 0$ in $H^2(T^2, \mathbb{R}).$ Since $H^2(T^2;\mathbb Z)$ is torsion-free, we derive that the pull-back of the Euler class of $\pi : S^5 \to \mathbb{C}P^2$ vanishes in $H^2(T^2,\mathbb{Z})$. Classification of circle bundles therefore gives
\begin{equation}\label{eq:hopf-topology}
 \widehat\Sigma\cong T^3,\qquad
 \sum_{j=0}^3 b_j(\widehat\Sigma;\R)=8.
\end{equation}

Now we regard $\widehat F:\widehat \Sigma\rightarrow \mathbb{S}^5\subset \mathbb{R}^6$ as an immersed submanifold in $\mathbb{R}^6$. Given a point $p\in \widehat \Sigma$, the unit $2$-sphere $\mathbb{S}^2$ on the normal space $N_p \widehat \Sigma$ can be parameterized as 
$$\{t z+ u N_1+vN_2\mid t^2+u^2+v^2=1\}.$$
With respect to the normal vector $t z+u N_1+v N_2$, the shape operator of $\widehat{F}$ is given by
$-t\Id+uA_1+vA_2$, where $A_1$ and $A_2$ take the form \eqref{eq:shape-matrices}. 
Set 
\begin{equation}\label{eq:I-def}
 \mathcal I(p)\triangleq
 \int_{\mathbb{S}^2}
    |\det(t\Id-uA_1-vA_2)|\,\dd S.
\end{equation}
Note that $\mathcal{I}(p)$ is constant along the Hopf fibers of $\widehat{\Sigma}$ and hence depends only on the base point $q=\pi(p)$. 
So the total absolute curvature of $\widehat{F}:\widehat{\Sigma}\to\mathbb{R}^6$ defined in \eqref{eq-tauM} can be reformulated as
\[
\tau(\widehat{\Sigma})=\frac{2\pi}{\pi^3}\int_{T^2}\mathcal{I}(q)\,dA,
\]
where we have used the facts that 
each fiber has length $2\pi$, and $\operatorname{Vol}(\mathbb{S}^5) = \pi^3$.  
Combining this with \eqref{eq:chern-lashof} and \eqref{eq:hopf-topology}, we obtain 
\begin{equation}\label{eq:I-lower}
 \int_{T^2}\mathcal I(q)\,d A\geq4\pi^2.
\end{equation}

Next, we attempt to derive an upper bound for $\mathcal{I}(q)$ in terms of $|\mathring A_1|^2+|\mathring A_2|^2$
, from which a lower bound of the functional $\mathcal{W}^-$ naturally arises by \eqref{eq-norm} and \eqref{eq:I-lower}. 

Note that the first columns of $\Id$, $A_1$, and $A_2$ yield the standard orthonormal basis of $\mathbb{R}^3$, and the third column of $A_1$ coincides with the second column of $A_2$. Using these relations, it is straightforward to verify that
\begin{equation}\label{eq:column-identity}
  t\Id-uA_1-vA_2=[x,A_1x,A_2x],
\end{equation}
where $x=(t,-u,-v)^{\mathsf T}$. 
Consequently, we obtain  
\begin{equation}\label{eq:I-determinant}
 \mathcal I(q)=\int_{\mathbb{S}^2}|\det[x,A_1x,A_2x]|\,\dd S=\int_{\mathbb{S}^2}|\det[x,\mathring A_1x,\mathring A_2x]|\,\dd S,
\end{equation}
where $\mathring A_1=A_1-\frac{2h}{3}\Id$ and $\mathring A_2=A_2$ have been used. This form of $\mathcal{I}(p)$ enables us to establish an upper bound in the next section.

\section{Proof of Montiel-Urbano's conjecture in the Lagrangian setting}\label{sec-proof}


Let $\mathrm{Sym}_0(3)$ denote the space of traceless symmetric $3\times 3$ real matrices, endowed with the Frobenius inner product $\langle P, Q \rangle = \operatorname{tr}(PQ)$. In this section, we first prove the following integral inequality for matrices in $\mathrm{Sym}_0(3)$. 
\begin{theorem}
\label{thm:matrix}
For every pair $P,Q\in\mathrm{Sym}_0(3)$, we have
\begin{equation}\label{eq:matrix-inequality}
 \int_{\mathbb{S}^2}\bigl|\det[x,Px,Qx]\bigr|\,dS
 \leq \sqrt{3}\,|P\wedge Q|,
\end{equation}
where
\[
|P\wedge Q|\triangleq \sqrt{\operatorname{tr}(P^2)\operatorname{tr}(Q^2)-\operatorname{tr}(PQ)^2}.
\]
Equality holds if and only if $[P,Q]=0$.
\end{theorem}
To prove it, in Subsection~\ref{sub-3.1} we interpret the determinant integral as the area with multiplicity of a quadratic map, thereby reducing the problem to a trigonometric integral inequality, whose proof is given in Subsection~\ref{sub-3.2}. Finally, in Subsection~\ref{sub-3.3}, applying the integral inequality \eqref{eq:matrix-inequality} to \eqref{eq:I-determinant} and invoking the arithmetic--geometric mean inequality yields an upper bound for $\mathcal{I}(p)$ in terms of $|\mathring{A}_1|^2+|\mathring{A}_2|^2$. Consequently, Montiel-Urbano's conjecture in the Lagrangian setting follows from \eqref{eq-norm} and \eqref{eq:I-lower}.

Set 
$$\mathcal D(P,Q)\triangleq \int_{\mathbb{S}^2}\bigl|\det[x,Px,Qx]\bigr|\,dS.$$
Note that for any linearly independent pair $\{P,Q\}$, one can choose an orthonormal basis $\{\widehat{P},\widehat{Q}\}$ of $\mathrm{Span}\{P,Q\}$ and write $P=a_1\widehat{P}+b_1\widehat{Q}$ and $Q=a_2\widehat{P}+b_2\widehat{Q}$. Then
\begin{equation}\label{eq:matrix-gram-schmidt}
 \det[x,Px,Qx]=(a_1b_2-a_2b_1)\det[x,\widehat{P}x,\widehat{Q}x].
\end{equation}
Since
\[
P\wedge Q=(a_1b_2-a_2b_1)\widehat{P}\wedge\widehat{Q}
\quad\text{and}\quad
|\widehat{P}\wedge\widehat{Q}|=1,
\]
it suffices to prove Theorem~\ref{thm:matrix} for orthonormal pairs.
\subsection{An geometric interpretation of the determinant integral $\mathcal{D}(P,Q)$}\label{sub-3.1}
Let $P,Q\in\Sym_0(3)$ be two matrices that are orthonormal with respect to the Frobenius inner product. Consider the map
\[
 q:\mathbb{S}^2\longrightarrow\R^2,\qquad
 q(x)=(x^TPx,x^TQx).
\]
Note that for an oriented orthonormal tangent basis $v, w$ on $\mathbb{S}^2$ with $v \times w = x$, the  Jacobian of $q$ is
\begin{equation}\label{eq-det}
\det \begin{pmatrix} 2v \cdot Px & 2w \cdot Px \\ 2v \cdot Qx & 2w \cdot Qx \end{pmatrix} = 4\langle v\wedge w, Px\wedge Qx\rangle=4 \det[x, Px, Qx].
\end{equation}
Set
\[
 \mathcal J(x)=\det[x,Px,Qx],\qquad
 \eta=\frac12(y_1\,d y_2-y_2\,d y_1).
\]
Then we have   
\begin{equation}\label{eq-pullback}
 q^*(d\eta)=4\mathcal J\,\dd S,
\end{equation}
where $\dd S$ is the standard area form of $\mathbb{S}^2$. 
It follows that 
\begin{equation*}
4\mathcal D(P,Q)
=4\int_{\mathbb{S}^2}\bigl|\det[x,Px,Qx]\bigr|\,\dd S=\int_{\mathbb{S}^2}\mathrm{sign}(\mathcal{J})\,q^*(d\eta),
\end{equation*}
which is the area of $q$ counted with multiplicity. Set  
$$\Omega_+\triangleq\{x\in \mathbb{S}^2:\mathcal J(x)>0\},~~~~~~
Z\triangleq\{x\in \mathbb{S}^2:\mathcal J(x)=0\}.$$ Since $\mathcal J(-x)=-\mathcal J(x)$, it follows that 
\begin{equation}\label{eq:spectral-green}
 \mathcal D(P,Q)
 =2\int_{\Omega_+}\mathcal J\,d S=\frac{1}{2}\int_{\Omega_+}q^*(d\eta).
\end{equation}
It is natural to rewrite the above integral over $\Omega_+$ as an integral over its boundary $Z$. To this end, we need to understand the geometry of $Z$. 
Note that, if $x\in Z$, by \eqref{eq-det} the tangential projections of $Px$ and $Qx$ onto
$T_x\mathbb{S}^2$ are linearly dependent. Thus,  some nonzero linear 
combination of $P$ and $Q$ has $x$ as an eigenvector, which means that there exits some $\theta\in [0, 2\pi]$ and $\lambda\in \mathbb{R}$ such that 
\begin{equation}
\label{eq-xtheta}(\cos\theta \,P+\sin\theta\,Q) x=\lambda x.
\end{equation}
Conversely, 
consider the pencil associated to $P$ and $Q$, 
\[
 C_\theta\triangleq \cos\theta \,P+\sin\theta\,Q,
 \qquad 0\leq\theta\leq2\pi,
\]
and denote its ordered eigenvalues by
$\lambda_{\min}(\theta)\leq\lambda_{\mathrm{mid}}(\theta)\leq\lambda_{\max}(\theta)$. Then it follows from 
\begin{equation}\label{eq-detctheta}
\mathcal J(x)=\det[x,Px,Qx]=\det[x,C_\theta\, x,C_\theta'\, x]
\end{equation}
that the six unit eigenvectors of $C_\theta$ all lie on $Z$. 
\begin{lemma} \label{lem-Lipsch}
Let $P,Q\in\mathrm{Sym}_0(3)$ be two matrices that are orthonormal with respect to the Frobenius inner product. 
For each $j \in \{\min, \mathrm{mid}, \max\}$, the function $\lambda_j(\theta)$ is Lipschitz continuous in $\theta$, and hence $\lambda_j\in W^{1,2}([0,2\pi])$.   
\end{lemma}
\begin{proof}
    By assumption, we have 
\[
    \|C_\theta'\|^2 = \|{-\sin\theta\,P + \cos\theta\,Q}\|^2 = \sin^2\theta \|P\|^2 + \cos^2\theta \|Q\|^2 = 1,
\]
which implies 
\[
    \|C_{\theta_1} - C_{\theta_2}\| \leq |\theta_1 - \theta_2|, \quad \forall \theta_1, \theta_2 \in [0, 2\pi].
\]
Then, it follows from Weyl's perturbation inequality \cite{horn2012matrix} that the ordered eigenvalue $\lambda_j$ satisfies 
\[
    |\lambda_j(\theta_1) - \lambda_j(\theta_2)| \leq \|C_{\theta_1} - C_{\theta_2}\|_2 \leq \|C_{\theta_1} - C_{\theta_2}\| \leq |\theta_1 - \theta_2|, 
\]
where $\|~\|_2$ is the spectral norm of matrices. 
This shows that $\lambda_j(\theta)$ is a Lipschitz continuous function on $[0, 2\pi]$. 

Consequently, $\lambda_j(\theta)$ is absolutely continuous. By Rademacher's theorem, its classical derivative $\lambda_j'(\theta)$ exists almost everywhere and is essentially bounded, implying $\lambda_j' \in L^\infty([0, 2\pi])\subset L^2([0, 2\pi])$. This completes the proof. 
\end{proof}

\begin{proposition}
\label{lem:spectral}
Under the same conditions as in Lemma~\ref{lem-Lipsch}, 
define 
\begin{equation}\label{eq:spectral-areas}
 \mathcal A_j\triangleq \frac12\int_0^{2\pi}
       \bigl(\lambda_j^2-(\lambda_j')^2\bigr)\,\dd\theta,
 \qquad j\in\{\min,\mathrm{mid},\max\},
\end{equation}
Then we have 
\begin{equation}\label{eq:spectral-area-identity}
 2\mathcal D(P,Q)
 =\mathcal A_{\max}+\mathcal A_{\min}-\mathcal A_{\mathrm{mid}}.
\end{equation}
Moreover, $\mathcal A_{\max}=\mathcal A_{\min}$. 
\end{proposition}

\begin{proof}
We first assume that $C_\theta$ has simple spectrum for every
$\theta$ and that $P,Q$ have no common eigenvector. 
\vskip 0.2cm
{\bf Claim.} {\em Under this assumption, $Z$ is a smooth curve in $\mathbb{S}^2$.}
\vskip 0.2cm


Given a point $x\in Z$, since $P$ and $Q$ share no common eigenvector, 
the parameter $\theta$ determined in \eqref{eq-xtheta} is unique  modulo $\pi$. 
Differentiating \eqref{eq-detctheta}, we obtain that for every $v\in T_x \mathbb{S}^2$, 
\begin{align}\label{eq-dJ}
 d\mathcal J_x(v)
 =\det[x,(C_\theta-\lambda \Id)v,C_\theta'\,x].
\end{align}
By assumption, $\lambda$ is a simple eigenvalue and $C_\theta' x\nparallel x$. It follows that 
$$(C_\theta-\lambda \Id) T_x \mathbb{S}^2= T_x \mathbb{S}^2. $$
So there exists a vector $v\in T_x \mathbb{S}^2$, such that $d\mathcal J(v)\neq 0$. This means $d\mathcal J\neq0$ on $Z$, i.e., $0$ is a regular value of $\mathcal{J}$.  Consequently, 
$Z=\partial \Omega_+$ is a smooth curve on $\mathbb{S}^2$.

It follows from 
\eqref{eq:spectral-green} that 
\begin{equation}\label{eq-ZI}
 \mathcal D(P,Q)
=\frac{1}{2}\int_{\partial\Omega_+}q^*\eta.
\end{equation}
Observe that $Z=\partial \Omega_+$ can be decomposed into three branches $\{Z_{\max}, Z_{\min}, Z_{\mathrm{mid}}\}$ according to the corresponding eigenvalues $\{\lambda_{\max}, \lambda_{\min}, \lambda_{\mathrm{mid}}\}$. 
Denote by  $x_j=x_j(\theta)$ the locally chosen 
unit eigenvector of
$C_\theta$ associated with the eigenvalue $\lambda_j$. This yields a local parameterization  of $Z_j$.  
Set
\[
 L_j\triangleq(C_\theta-\lambda_j\Id)|_{T_{x_j} \mathbb{S}^2},
 \qquad w_j\triangleq\operatorname{proj}_{x_j^\perp}(C_\theta'x_j).
\]
By assumption, the operator $L_j$ is invertible and  $w_j\neq0$. 
Differentiating the  equation $(C_\theta-\lambda_j\Id)x_j=0$ 
and projecting it onto $T_{x_j} \mathbb{S}^2$  yields 
\begin{equation}\label{eq:eigenvector-derivative}
 w_j=(C_\theta'-\lambda_j'\Id)x_j=-L_jx_j'. 
\end{equation}
Choose $v\in T_{x_j}\mathbb{S}^2$ such that $x_j'\wedge v$ points in the outward normal direction of $\mathbb{S}^2$ at $x_j$. 
It follows from \eqref{eq-dJ} that 
\begin{align}
 d\mathcal J_x(v)
 =\det[x,L_j v,w_j]=\det[x,L_j x_j',L_j v]=\langle x, L_j x_j'\wedge L_j v\rangle=\det(L_j)\,\langle x,  x_j'\wedge v\rangle.   \label{eq:critical-orientation}
\end{align}
Set $\sigma_j\triangleq\operatorname{sign}\det L_j$. By definition,  we have 
\[
 \qquad\sigma_{\max}=\sigma_{\min}=1,
 \qquad\sigma_{\mathrm{mid}}=-1.
\]
Then, \eqref{eq:critical-orientation} implies that increasing $\theta$ gives the positive boundary
orientation for $Z_{\max}$ and $Z_{\min}$ 
and
the opposite orientation for $Z_\mathrm{mid}$. An illustrative example is shown in Figure~\ref{fig:both}. In the left picture, the blue, red, and green curves denote \(Z_{\mathrm{mid}}\), \(Z_{\max}\), and \(Z_{\min}\), respectively, and the brown domain denotes \(\Omega_+\).
\begin{figure}
    \centering
    \begin{subfigure}{0.43\linewidth}
        \centering
        \includegraphics[width=\linewidth]{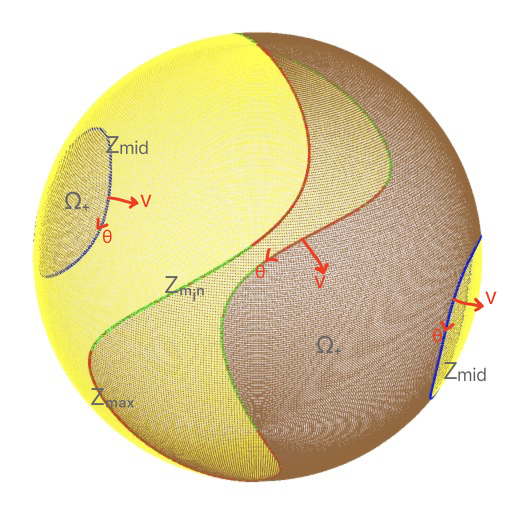}
        \caption{The domain $\Omega_+$ and its boundary $Z=\partial\Omega_+$}
        \label{fig:critical}
    \end{subfigure}
    \hfill
    \begin{subfigure}{0.43\linewidth}
        \centering
        \includegraphics[width=\linewidth]{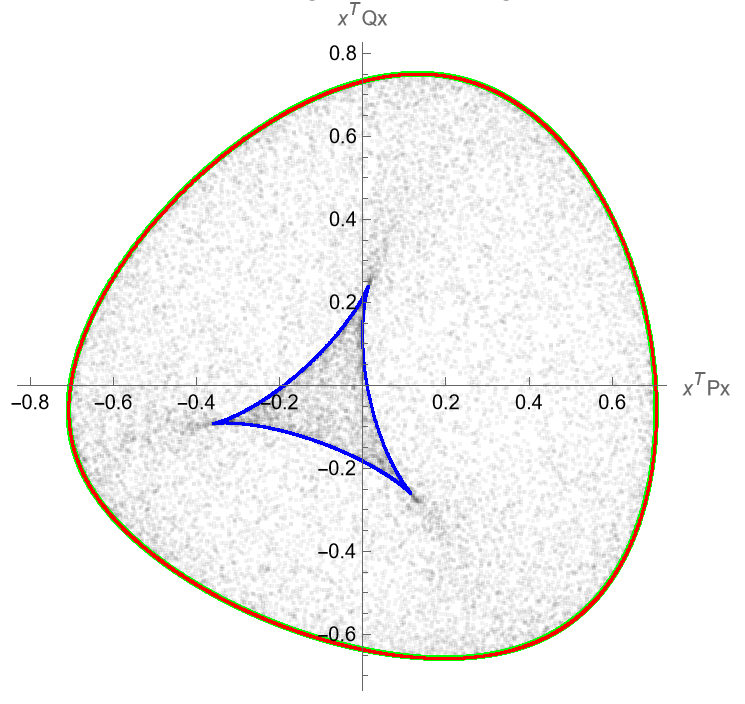}
        \caption{The image curve $q(Z)$}
        \label{fig:imagecurve}
    \end{subfigure}
    \caption{An example given by $P=
\frac{1}{\sqrt{2}}\begin{pmatrix}
 -1 & 0 & 0 \\
 0 & 0 & 0 \\
 0 & 0 & 1 \\
\end{pmatrix}
$ and 
$Q=
\frac{1}{3 \sqrt{30}}\begin{pmatrix}
-1 & 2 & 8 \\
 2 & 2 & 8 \\
 8 & 8 & -1 \\
\end{pmatrix}
$ 
    }
    \label{fig:both}
\end{figure}

Note that if we  track $\pm x_j(\theta)$ on the whole interval $[0,2\pi]$, then its  image is a twofold covering of $Z_j$ and $Z_{\max}=Z_{\min}$ as sets (since $\lambda_{\max}(\theta+\pi)=-\lambda_{\min}(\theta)$). For each $j\in\{\max,\min,\mathrm{mid}\}$, let $Z_j^+$ denote
the oriented $1$-cycle obtained by tracing both $x_j(\theta)$
and $-x_j(\theta)$ for $0\leq\theta\leq2\pi$, with the
parameterization multiplicities retained and both paths
oriented by increasing $\theta$. Consequently, we obtain,  as chains, 
$$2\partial\Omega_{+}=Z_{\max}^++Z_{\min}^+-Z_{\mathrm{mid}}^-,$$
where $\partial\Omega_{+}$ is endowed with the  positive boundary  orientation. 
Set $$\gamma_j(\theta)\triangleq q(x_j(\theta)), ~~~\theta\in [0, 2\pi], ~~~j\in\{\max, \min, \mathrm{mid}\}.$$
Then it follows from \eqref{eq-ZI} that 
\begin{equation}\label{eq-ZI2}
\begin{split}
 2\mathcal D(P,Q)
&=\int_{\partial\Omega_+}q^*\eta=\frac12\left(\int_{Z_{\max}^+}q^*\eta+\int_{Z_{\min}^+}q^*\eta-\int_{Z_{\mathrm{mid}}^+}q^*\eta\right)\\
&=\int_{\gamma_{\max}^+}\eta+\int_{\gamma_{\min}^+}\eta-\int_{\gamma_{\mathrm{mid}}^+}\eta\\
&=\int_0^{2\pi}\gamma_{\max}^*\,\eta+\int_0^{2\pi}\gamma_{\min}^*\,\eta-\int_0^{2\pi}\gamma_{\mathrm{mid}}^*\,\eta,
\end{split}
\end{equation}
where we have used $q(x_j)=q(-x_j)$ in the third equality. 
We point out that here we do not require the critical-value curves
$\gamma_j$ to be immersed. In particular, cusps of such curves
cause no difficulty.

Substituting 
\begin{equation}\label{eq-PQ}
P=\cos\theta\, C_\theta-\sin\theta\, C_\theta',\qquad  Q=\sin \theta \,C_\theta+\cos\theta\, C_\theta'
\end{equation} 
into the definition of $\gamma_j$, and applying \eqref{eq:eigenvector-derivative} yields 
\[
 \gamma_j=\lambda_jn_\theta+\lambda_j'n_\theta^\perp,
\]
where $n_\theta=(\cos\theta,\sin\theta)$ and
$n_\theta^\perp=(-\sin\theta,\cos\theta)$. 
Integration by parts then gives
\begin{equation}\label{eq:spectral-envelope-integral}
 \begin{split}
 \int_0^{2\pi}\gamma_j^*\eta
 =\frac12\int_0^{2\pi}
       \lambda_j(\lambda_j+\lambda_j'')\,d\theta=\frac12\int_0^{2\pi}
       \bigl(\lambda_j^2-(\lambda_j')^2\bigr)\,d\theta
 =\mathcal A_j.
 \end{split}
\end{equation}
Equations~\eqref{eq-ZI2} and  \eqref{eq:spectral-envelope-integral}
prove \eqref{eq:spectral-area-identity} under the temporary
assumptions. Moreover, 
$\lambda_{\max}(\theta+\pi)=-\lambda_{\min}(\theta)$ implies
$\mathcal A_{\max}=\mathcal A_{\min}$ directly.

To complete the proof for every orthonormal pair $\{P, Q\}$, we show  that the orthonormal pairs which share no common eigenvector and  whose pencil $C_\theta=\cos \theta \,P+\sin\theta\,Q$ has no repeated eigenvalue for every $\theta\in [0,2\pi]$ are dense in the space of ordered
orthonormal pairs, which forms a 7-dimensional Stiefel manifold
$V_2(\Sym_0(3))$.

For a unit-norm matrix in $\Sym_0(3)$, if it has a repeated eigenvalue, then all eigenvalues are determined by the trace-free and unit-norm conditions. In fact, the repeated-eigenvalue locus $\mathcal{V}$ among unit matrices in
$\Sym_0(3)$ can be parameterized as follows, 
\[
 \mathcal V=\left\{
    \pm\sqrt{\frac32}\left(xx^T-\frac13\Id\right):
                  [x]\in\RP^2\right\}.
\]
It is the disjoint union of two copies of $\RP^2$. 
It follows from \eqref{eq-PQ} that any orthonormal pair $\{P,Q\}$ whose pencil has a repeated eigenvalue at some angle $\theta$ can be determined by a unit  $C\in\mathcal  V$ and a unit matrix 
$Y\in C^\perp\subset \Sym_0(3)$. 
The data $\{\theta, C, Y\}$ forms a compact smooth parameter
space of dimension $1+2+3=6$. Its image in $V_2(\Sym_0(3))$ is closed
and, by Sard's theorem, has measure zero. The complement is
therefore open and dense. 

For fixed $[x]\in\RP^2$, the matrices
with $ x$ as eigenvectors form a 3-dimensional subspace of
$\Sym_0(3)$. So the space of pairs with a common eigenvector are consequently the image of a compact parameter bundle of dimension $2+\dim V_2(\mathbb{R}^3)=5$ 
Its complement is likewise open and
dense. 

Consequently, as the intersection of the aforementioned complements, the orthonormal pairs which share no common eigenvector and  whose pencil $C_\theta=\cos \theta \,P+\sin\theta\,Q$ has no repeated eigenvalue for every $\theta\in [0,2\pi]$ are open and dense in $V_2(\Sym_0(3))$. 

Finally, we show that the functionals defined in \eqref{eq:spectral-areas} are continuous on $V_2(\Sym_0(3))$. For an arbitrary orthonormal pair $\{P,Q\}$, 
since $\tr C_\theta = 0$ and $\tr(C_\theta^2) = 1$, the characteristic polynomial of $C_\theta = \cos\theta\,P + \sin\theta\,Q$ is given by
\begin{equation}\label{eq:pencil-characteristic}
    \lambda^3 - \frac{1}{2}\lambda - \det C_\theta.
\end{equation} 
Set
\begin{equation}\label{eq:normalized-cubic}
    f(\theta) \triangleq 3\sqrt{6}\det C_\theta.
\end{equation}
By considering the first derivative of \eqref{eq:pencil-characteristic} with respect to $\lambda$, it is straightforward to verify that this polynomial admits a repeated eigenvalue if and only if $|f(\theta)| = 1$. 
Since $f$ is a trigonometric polynomial and $f(\theta+\pi)=-f(\theta)$, $1 - f(\theta)^2$ cannot be identically zero. Consequently, there are only finitely many angles $\theta \in [0, 2\pi)$ at which $C_\theta$ has repeated eigenvalues. 

Let $\{P^{(k)}, Q^{(k)}\}$ be a sequence of orthonormal pairs converging to $\{P, Q\}$. It follows naturally that the ordered eigenvalues $\lambda_j^{(k)}$ converge uniformly to $\lambda_j$. Note that whenever $\lambda_j^{(k)}$ is a simple eigenvalue, the adjugate matrix of $C_\theta^{(k)} - \lambda_j^{(k)} I$ is non-zero, and any of its non-zero columns yields an eigenvector corresponding to $\lambda_j^{(k)}$. Thus, at any angle where the limiting matrix has a simple eigenvalue, the corresponding spectral projection $\Pi_j^{(k)} \triangleq x_j^{(k)}{x_j^{(k)}}^T$ converges (as does the eigenvector $x_j^{(k)}$, given an appropriate choice of sign). Then, by \eqref{eq:eigenvector-derivative}, 
\[
    (\lambda_j^{(k)})' = \tr\big(\Pi_j^{(k)}(C_\theta^{(k)})'\big) = \langle \Pi_j^{(k)},(C_\theta^{(k)})'\rangle
\]
also converges pointwise to $\lambda_j'$. 
Furthermore, we have the following uniform bound for all $k$, 
\[
    |(\lambda_j^{(k)})'| \leq \|(C_\theta^{(k)})'\|
    = 1.
\]
Then, by the Dominated Convergence Theorem, the functionals $\mathcal{A}_j$ defined in \eqref{eq:spectral-areas} are continuous.
Observe also that the determinant integral  $\mathcal{D}(P,Q)$ is continuous. Therefore, we can extend the identity \eqref{eq:spectral-area-identity} 
from the open dense set to every orthonormal
pair. 
\end{proof}

Next, we seek an exact expression for $\mathcal{A}_j$ by solving for the eigenvalues of $C_\theta$ via \eqref{eq:pencil-characteristic}. Computing the discriminant of \eqref{eq:pencil-characteristic} and using the fact that \(A\) has three real eigenvalues, we conclude that \(|f|\leq1\). 

We 
set
\[
 \psi\triangleq \frac13\arcsin f\in[-\pi/6,\pi/6].
\]
Then the ordered eigenvalues can be written as 
\begin{equation}\label{eq:ordered-eigenvalue-parametrization}
\lambda_{\mathrm{mid}}=-\sqrt{\frac23}\sin\psi,\qquad \lambda_{\max}=-\sqrt{\frac23}\sin(\psi-2\pi/3),\qquad
 \lambda_{\min}=-\sqrt{\frac23}\sin(\psi+2\pi/3).
\end{equation}
Substituting them into Proposition~\ref{lem:spectral} gives the exact expression
\begin{equation}\label{eq:matrix-trig-expression}
 2\int_{\mathbb{S}^2}\bigl|\det[x,Px,Qx]\bigr|\,d S
 =2\mathcal D(P,Q)=\frac16\int_0^{2\pi}
 \left[1+2\cos(2\psi)
      +\bigl(2\cos(2\psi)-1\bigr)(\psi')^2\right]\dd\theta,
\end{equation}
where the derivative is understood almost everywhere. 
In the next subsection, we will establish an upper bound for the trigonometric integral in \eqref{eq:matrix-trig-expression} over all functions of the form $\psi=\frac13\arcsin f$, where
\[
f\in \operatorname{span}\{\cos\theta,\sin\theta,\cos3\theta,\sin3\theta\},~~~\|f\|_\infty\leq1.
\]
This class includes $f=3\sqrt6\det(P\cos\theta+Q\sin\theta)$, which is a trigonometric polynomial of order $3$ satisfying $f(\theta+\pi)=-f(\theta)$.

\subsection{A trigonometric integral inequality}
\label{sub-3.2}
\begin{proposition}\label{lem:trig}
Given a function $f\in \operatorname{span}\{\cos\theta,\sin\theta,\cos3\theta,\sin3\theta\}$ with $\|f\|_\infty\leq1$. Set $\psi=\frac13\arcsin f$. 
Then the integral \begin{equation}\label{eq:trig-inequality} \mathcal T(f)\triangleq \frac16\int_0^{2\pi}
 \left[1+2\cos(2\psi)
       +\bigl(2\cos(2\psi)-1\bigr)(\psi')^2\right]d\theta\leq2\sqrt3,
\end{equation}
with equality if and only if
$f(\theta)=\cos(3\theta-\theta_0)$ for some $\theta_0\in\R$. Here the derivative $\psi'$ is understood almost everywhere.
\end{proposition}
We will prove this proposition by a variational approach.
\begin{lemma}\label{lem-2nd}
Given 
a function $\xi\in\mathrm{Span}\{\cos\theta, \sin\theta\}$. 
\begin{enumerate}
    \item[{\rm(1)}] For any $f\in \operatorname{Span}\{\cos3\theta, \sin3\theta\}$ with $\|f\|_\infty<1$, the first variation $D\mathcal{T}_f[\xi]$ of $\mathcal{T}(f)$ along the direction determined by $\xi$ vanishes.
 \item[{\rm(2)}] For any $f\in \operatorname{Span}\{\cos\theta,\sin\theta, \cos3\theta, \sin3\theta\}$ with $\|f\|_\infty<1$, the second variation $D^2\mathcal T_f[\xi,\xi]$
 of $\mathcal{T}(f)$ along the direction determined by $\xi$ satisfies 
\begin{equation}\label{eq:trig-concavity}
 D^2\mathcal T_f[\xi,\xi]
 \leq-\frac1{216}\int_0^{2\pi}\xi^2\,d\theta.
\end{equation}
\end{enumerate}
\end{lemma}
\begin{proof}
To perform the variational calculation, we rewrite the integrand $\mathcal T(f)$ as a function of
$f$ and $f'$. Introduce 
\[
 y=\sin\psi\in\left(-\frac12,\frac12\right),\qquad
 f=3y-4y^3,\qquad \frac{d f}{d y}=3(1-4y^2).
\]
The identities
\[
 2\cos(2\psi)-1=1-4y^2,\qquad
 1-f^2=(1-y^2)(1-4y^2)^2, \qquad \psi' =  \frac{f'}{3\sqrt{1-f^2}}
\]
transform $\mathcal{T}(f)$ 
into 
\begin{equation}\label{eq:trig-vw}
 \begin{split}
 \mathcal T(f)&=\int_0^{2\pi}
                  \bigl[V(f)+U(f)(f')^2\bigr]d\theta,
 \end{split}
\end{equation}
where 
\begin{equation}\label{eq-UV}
V(f)=\frac12-\frac23y^2,\qquad
 U(f)=\frac{1}{54(1-y^2)(1-4y^2)}.
 \end{equation}
Taking the derivative with respect to $f$  
yields 
\begin{align}
 V'(f)=-\frac{4y}{9(1-4y^2)},\qquad 
 V''(f)=-\frac{4(1+4y^2)}{27(1-4y^2)^3}, \qquad  U'(f)=\frac{y(5-8y^2)}{81(1-y^2)^2(1-4y^2)^3}. 
       \label{eq:v-derivatives}
\end{align}

Using the above notation, a direct computation yields 
\begin{equation}\label{eq:trig-first-variation}
 D\mathcal{T}_f[\xi]
 =\left.\frac{d}{dt}\right|_{t=0}\mathcal{T}(f+t\xi)=\int_0^{2\pi}
   \bigl[V'(f)-U'(f)(f')^2-2U(f)f''\bigr]\xi\,d\theta.
\end{equation}
When $f$ belongs to $\operatorname{Span}\{\cos3\theta, \sin3\theta\}$, write $f=k\cos (3\theta-\theta_0)$ with $0\leq k<1$, we  have 
$$f''=-9f,\qquad(f')^2=9(k^2-f^2).$$
Consequently, using \eqref{eq-UV} and \eqref{eq:v-derivatives}, we obtain 
\begin{equation}\label{eq:trig-amplitude-identity}
V'(f)-U'(f)(f')^2-2U(f)f''= V'(f)-9(k^2-f^2)U'(f)+18fU(f)=9(1-k^2)U'(f),
\end{equation}
which has period $2\pi/3$ and is
therefore orthogonal to $\xi$. 
This completes the proof of the first conclusion.


For the second conclusion, it is straightforward to verify that 
\[
 D^2\mathcal T_f[\xi,\xi]=\left.\frac{d^2}{dt^2}\right|_{t=0}\mathcal{T}(f+t\xi)
 =\int_0^{2\pi}
 \left[(V''+U''(f')^2)\xi^2
       +4U'f'\xi\xi'+2U(\xi')^2\right]d\theta.
\]
Periodicity and $\xi''=-\xi$ imply
\begin{align*}
 \int_0^{2\pi}4U'f'\xi\xi'\,d\theta
 &=-2\int_0^{2\pi}
       \bigl(U''(f')^2+U'f''\bigr)\xi^2\,d\theta,\\
 \int_0^{2\pi}2U(\xi')^2\,d\theta
 &=\int_0^{2\pi}
       \bigl(2U+U''(f')^2+U'f''\bigr)\xi^2\,d\theta.
\end{align*}
It follows that 
\begin{equation}\label{eq-trig-second-variation}
 D^2\mathcal T_f[\xi,\xi]
 =\int_0^{2\pi}(V''+2U-U'f'')\xi^2\,\dd\theta.
\end{equation}

Now we write $f=f_1+f_3$, with $f_1\in \mathrm{Span}\{\cos\theta, \sin\theta\}$, and 
$$
 f_3=\mathfrak a\cos3\theta+\mathfrak b\sin3\theta=k\cos(3\theta-\theta_0),
 $$
 where $k=\sqrt{\mathfrak{a}^2+\mathfrak{b}^2}\geq0$. 
Then substituting $f''=-9f+8f_1$ into \eqref{eq-trig-second-variation} yields 
\begin{equation}\label{eq:trig-second-variation}
 D^2\mathcal T_f[\xi,\xi]=
\int_0^{2\pi}\big(D_0(f)-8f_1U'(f)\big)\xi^2\,\dd\theta, 
\end{equation}
where $
 D_0(f)\triangleq V''(f)+2U(f)+9fU'(f)$.  It follows from \eqref{eq-UV} and \eqref{eq:v-derivatives} that  
\begin{equation}\label{eq:trig-d0}
 D_0(f)=-\frac{3-4y^2}{27(1-y^2)^2(1-4y^2)}<0.
\end{equation}
 
Note that  
$$f_3(\theta) = \frac{1}{3}\left[f(\theta) + f\left(\theta + \frac{2\pi}{3}\right) + f\left(\theta + \frac{4\pi}{3}\right)\right],$$
which yields  
\begin{equation}\label{eq-f3norm}
 \|f_3\|_\infty\leq\|f\|_\infty\leq 1.   
\end{equation}
If $f_1f\geq0$, then $f_1U'(f)\geq0$, since $U'(f)$ and $f$ have the same sign. If $f_1f<0$, the identity
$f_3=f-f_1$ and \eqref{eq-f3norm} 
yield 
\[
 |f_1|+|f|=|f_3|\leq k\leq1.
\]
In either case, we obtain 
\begin{equation}\label{eq:trig-coefficient-bound}
 D_0(f)-8f_1U'(f)
 \leq D_0(f)+8(1-|f|)|U'(f)|.
\end{equation}
To bound the right-hand side, put $Y=|y|\in[0,1/2)$, so that
$|f|=3Y-4Y^3$. It is straightforward to verify that 
\begin{equation*}
 \begin{split}
 -D_0(f)-8(1-|f|)|U'(f)|
 &=\frac{\mathcal Q(Y)}
 {81(1-Y^2)^2(1-2Y)(1+2Y)^3}\\
 \end{split}
\end{equation*}
where
\[
 \mathcal Q(Y)=16Y^4+16Y^3-16Y^2-4Y+9.
\]
Note that, on $[0, \frac{1}{2}]$, 
$$Q(Y)\geq 9-16Y^2-4Y\geq 3,~~~~~~(1-Y^2)^2(1-2Y)(1+2Y)^3\leq 8,
 $$
which implies 
\begin{equation}\label{eq:trig-uniform-gap}
 \begin{split}
 -D_0(f)-8(1-|f|)|U'(f)|
 \geq\frac{3}{81\cdot8}=\frac1{216}.
 \end{split}
\end{equation}
Combining \eqref{eq:trig-second-variation} $\sim$
\eqref{eq:trig-uniform-gap},
we obtain \eqref{eq:trig-concavity}.
This completes the proof of the second conclusion. 
\end{proof}
\begin{proof}[Proof of Proposition~{\rm \ref{lem:trig}}]
As in the proof of Lemma~\ref{lem-2nd},  we write $f=f_1+f_3$, with $f_1\in \mathrm{Span}\{\cos\theta, \sin\theta\}$, and 
$$
 f_3=\mathfrak a\cos3\theta+\mathfrak b\sin3\theta=k\cos(3\theta-\theta_0).
 $$
 
We first assume $\|f\|_\infty<1$; passage to the boundary of the unit
ball will be justified below.  
It follows from \eqref{eq-f3norm} 
that the entire segment
$$f_3+sf_1=sf+(1-s)f_3, ~~~~~~0\leq s\leq1,$$ lies in the open unit ball.

Applying Lemma~\ref{lem-2nd} to $f_s=f_3+sf_1$ with $\xi=f_1$ yields 
$$\frac{d}{ds}\mathcal{T}(f_s)=\left.\frac{d}{ds}\right|_{s=0}\mathcal{T}(f_s)+\int_{0}^{s}\frac{d^2}{dt^2}\mathcal{T}(f_3+tf_1)dt\leq-\frac {s}{216}\int_0^{2\pi}\xi^2\,d\theta,~~~0\leq s\leq 1.$$
Integrating it with respect to $s$ from $0$ to $1$, we obtain
\begin{equation}\label{eq:trig-first-harmonic-loss}
 \mathcal{T}(f_3+f_1)
 \leq\mathcal{T}(f_3)-\frac1{432}\int_0^{2\pi}f_1^2\,d\theta,
\end{equation}
which implies that any $f$ with $\|f\|_\infty<1$ and $f_1\neq 0$ fails to maximize $\mathcal{T}(f)$ in the open unit ball. 


Next we first extend \eqref{eq:trig-first-harmonic-loss} to an arbitrary trigonometric polynomial $f$ with $\|f\|_\infty\leq1$. Given such an $f$, set $M\triangleq\|f''\|_\infty$. 
If $M=0$, then $f=0$, and there is nothing to prove. Otherwise,
at a given $\theta$ applying Taylor's formula  to $g=f$ (if $f(\theta)\geq 0$) or $g=-f$ (if $f(\theta)<0$), we have 
$$1\geq g(\theta + h) = g(\theta) + g'(\theta)h + \frac{1}{2}g''(\xi)h^2\geq g(\theta) + g'(\theta)h - \frac{M}{2}h^2 .$$
Then, chooseing  
$h=g'(\theta)/M$ yields  
\[
 (f'(\theta))^2=(g'(\theta))^2\leq 2M(1-g(\theta)) \leq2M(1-|f(\theta)|)
                    \leq2M(1-f(\theta)^2).
\]
For $0<\rho<1$, set $\psi_\rho=\frac13\arcsin(\rho f)$. Then  $\psi_\rho$ converge uniformly to
$\psi=\frac13\arcsin f$, and 
\begin{equation}\label{eq:trig-boundary-derivative}
 (\psi_\rho')^2
 =\frac{\rho^2(f')^2}{9(1-\rho^2f^2)}\leq\frac{2M}{9}.
\end{equation}
The uniform derivative bound implies that $\psi$ is Lipschitz continuous. Away from the finitely many points where $f = \pm 1$, $\psi$ has the standard derivative $\psi'$, and the sequence $\psi_\rho'$ converges pointwise to $\psi'$. Consequently, the integrands defining $\mathcal T(\rho f)$ are uniformly bounded and converge almost everywhere. 
Dominated convergence yields 
\begin{equation}\label{eq:trig-boundary-continuity}
 \lim_{\rho\to1}\mathcal T(\rho f)=\mathcal T(f).
\end{equation}
Applying \eqref{eq:trig-first-harmonic-loss} to $\rho f$ and letting
$\rho\to1$ therefore extends that inequality to the closed unit ball. Then, we derive that  any $f$ with $\|f\|_\infty\leq1$ and $f_1\neq 0$ fails to maximize $\mathcal{T}(f)$. 

It remains to consider the case of $f_1=0$. Note that we can assume $f=k\cos3\theta$, since  
translation in 
$\theta$ leaves $\mathcal T$ unchanged. 
For this family,
$\partial_kf=f/k$. If $0<k<1$, it follows from \eqref{eq:trig-amplitude-identity}
 that 
\begin{equation}\label{eq:trig-amplitude-monotonicity}
 \frac{d}{d k}\mathcal T(k\cos3\theta)
 =\frac{9(1-k^2)}{k}
   \int_0^{2\pi}fU'(f)\,d\theta>0. 
\end{equation}
At $k=1$, the function
$\psi=\frac13\arcsin(\cos3\theta)$ is a triangular wave with values
in $[-\pi/6,\pi/6]$ and $(\psi')^2=1$ almost everywhere. During one
period it traverses this interval six times. Hence
\begin{align*}
 \mathcal T(\cos3\theta)
 =\frac23\int_0^{2\pi}\cos(2\psi)\,d\theta
 =\frac23\cdot6\int_{-\pi/6}^{\pi/6}\cos(2t)\,d t
 =2\sqrt3.
\end{align*}
Combining \eqref{eq:trig-amplitude-monotonicity} with continuity at $k=0$, we conclude that \eqref{eq:trig-inequality} holds. The equality case is straightforward. This completes the proof.

\end{proof}

\begin{proof}[Proof of Theorem~\ref{thm:matrix}]
As mentioned before, we need only prove this theorem for orthonormal pairs, for which the matrix inequality \eqref{eq:matrix-inequality} for orthonormal pairs follows from \eqref{eq:matrix-trig-expression} and 
Proposition~\ref{lem:trig}.

Now suppose that equality in \eqref{eq:matrix-inequality} holds. Then equality in \eqref{eq:trig-inequality} also holds. By Proposition~\ref{lem:trig}, we have $f(\theta)=\cos(3\theta-\theta_0)$ for some $\theta_0$. 
Consequently $|\psi'|=1$ almost everywhere, and
\begin{equation}\label{eq:spectral-equality-norm}
 \sum_j(\lambda_j')^2=(\psi')^2
 =1=\tr\bigl((C_\theta')^2\bigr).
\end{equation}
At a simple-eigenvalue angle, express $C_\theta'$ in an orthonormal
eigenbasis of $C_\theta$. By \eqref{eq:eigenvector-derivative}, the diagonal entries of $C_\theta'$ are $\langle C_\theta' x_j, x_j\rangle=\lambda_j'$. Then 
\eqref{eq:spectral-equality-norm} implies that the sum of the squares
of all off-diagonal entries is zero. Hence $C_\theta$ and $C_\theta'$ can be  simultaneously diagonalized, which implies 
$[C_\theta,C_\theta']=0$. Since
$[C_\theta,C_\theta']=[P,Q]$, we obtain $[P,Q]=0$. 

Conversely, assume the orthonormal pair $\{P, Q\}$ satisfies $[P,Q]=0$, then $P$ and $Q$ 
can be simultaneously diagonalized, whose  diagonal vectors are denoted by $(p_1, p_2, p_3)$ and $(q_1, q_2, q_3)$, respectively. By the orthonormal condition of $\{P, Q\}$, these two vectors form 
an orthonormal basis of $(1,1,1)^\perp\subset\R^3$. Consequently, we have 
\[
 \det[x,Px,Qx]=\det \begin{pmatrix} x_1 & p_1 x_1 & q_1 x_1 \\ x_2 & p_2 x_2 & q_2 x_2 \\ x_3 & p_3 x_3 & q_3 x_3 \end{pmatrix}= x_1 x_2 x_3 \det \begin{pmatrix} 1 & p_1 & q_1 \\ 1 & p_2 & q_2 \\ 1 & p_3 & q_3 \end{pmatrix}= \pm\sqrt3\,x_1x_2x_3.
\]
It is straightforward to verify that 
\[
 \int_{\mathbb{S}^2}|x_1x_2x_3|\,d S=1,
\]
from which equality follows.
\end{proof}


\subsection{A lower bound for the Willmore-type functional $\mathcal{W}^-$}\label{sub-3.3}
With the determinant integral inequality established by Subsections~\ref{sub-3.1} and \ref{sub-3.2}, we can now prove Montiel-Urbano's conjecture in the Lagrangian setting. This follows from the theorem below.
\begin{theorem}\label{thm:lagrangian}
For every immersed Lagrangian torus $F:T^2\to\mathbb{C}P^2$, we have
\begin{equation}\label{eq-strong-bound}
 \mathcal{W}^-(F)\geq \frac{8\pi^2}{3\sqrt{3}}+\frac{1}{9}\int_{T^2}|H|^2\,dA\geq \frac{8\pi^2}{3\sqrt{3}},
\end{equation}
and equality holds if and only if $F$ is congruent to the Clifford torus in $\mathbb{C}P^2$.
\end{theorem}
\begin{proof}
We continue to use the notation of Section~\ref{sec:hopf}. Consider the Hopf bundle $\widehat \Sigma$ of $F$ and its immersion $\widehat F: \widehat \Sigma\rightarrow \mathbb{S}^5\subset \mathbb{R}^6$. Applying Theorem~\ref{thm:matrix} to \eqref{eq:I-determinant} yields the following pointwise upper bound of the absolute curvature $\mathcal{I}$ of $\widehat{F}$, 
\begin{equation}\label{eq-Ibound}
 \mathcal I=\int_{\mathbb{S}^2}|\det[x,A_1x,A_2x]|\,\dd S=\int_{\mathbb{S}^2}|\det[x,\mathring A_1x,\mathring A_2x]|\,\dd S\leq \sqrt{3}|\mathring A_1\wedge \mathring A_2|,
\end{equation}
where $A_1,A_2$ are of the form \eqref{eq:shape-matrices}, with $\mathring{A}_1,\mathring{A}_2$ their trace-free parts. 
Using the
arithmetic-geometric mean inequality, we futhure have 
\begin{equation}\label{eq-AG}
    |\mathring A_1\wedge \mathring A_2|=\sqrt{|\mathring A_1|^2|\mathring A_2|^2-\langle \mathring A_1, \mathring A_2 \rangle^2}\leq \frac{|\mathring A_1|^2+|\mathring A_2|^2}{2}.
\end{equation}
From \eqref{eq-Ibound}, \eqref{eq-AG} and \eqref{eq-norm}, we obtain 
$$ \mathcal I\leq 2\sqrt{3}(K+3|\phi|^2-\frac1{12}h^2).$$
Recall that, using the Chern-Lashof inequality, we have established the following lower bound in \eqref{eq:I-lower}:
\begin{equation*}
 \int_{T^2}\mathcal{I}(q)\,dA\geq4\pi^2.
\end{equation*}
Combining the above two inequalities with \eqref{eq-W-} and the Gauss-Bonnet formula, we obtain
\[
\mathcal{W}^-(F)\geq \frac{8\pi^2}{3\sqrt{3}}+\frac{1}{9}\int_{T^2}|H|^2\,dA.
\]

Suppose equality holds in \eqref{eq-strong-bound}. Then 
$H=0$ and 
equality also holds in \eqref{eq-Ibound} and \eqref{eq-AG}. 
By Theorem~\ref{thm:matrix}, we have $[A_1,A_2]=[\mathring A_1,\mathring A_2]=0$. Direct computation gives
\begin{equation}\label{eq:commutator-curvature}
 [A_1,A_2]=\begin{pmatrix}0&0&0\\0&0&K\\0&-K&0\end{pmatrix}. 
\end{equation}
So $K=0$. The conclusion follows from the classification of minimal Lagrangian flat surfaces in \cite{Ludden-Okumura-Yano}. 
\end{proof}
\begin{rem} If one wishes to avoid the integral of $|\phi|^2$, \eqref{eq-norm} can also be reformulated by \eqref{eq:gauss} as
\begin{equation*}\label{eq-norm2}
|\mathring{A}_1|^2+|\mathring{A}_2|^2=4+4|\phi|^2+\frac5{3}h^2=6-2K+\frac8{3}h^2.
\end{equation*}
Then \eqref{eq-Ibound} and \eqref{eq-AG} yield the following inequality for the area of $F$, 
\begin{equation}\label{eq:strong-bound}
 \mathcal{A}(F)+\frac49 \int_{T^2}|H|^2\,\dd A\geq \frac{4\pi^2}{3\sqrt3}=\mathcal{A}(T_{\CL}),
\end{equation}
from which
\begin{equation*}\label{eq:wminus-bound}
 \mathcal{W}^-(F)=2\mathcal{A}(F)+ \int_{T^2}|H|^2\,dA\geq 2\mathcal{A}(T_{\CL})+ \frac19\int_{T^2}|H|^2\,\dd A\geq2\mathcal{A}(T_{\CL})=\mathcal{W}^-(T_{\CL}).
\end{equation*}
Using \eqref{eq:gauss}, we also obtain the following lower bound of  the whole Willmore functional for closed Lagrangian surfaces in $\mathbb{C}P^2$, 
$$\mathcal{W}(F)=\mathcal{A}(F)+\int_{T^2}|H|^2\,\dd A\geq \mathcal{A}(T_{\CL})+\frac59\int_{T^2}|H|^2\,\dd A\geq \mathcal{A}(T_{\CL})=\mathcal{W}(T_{\CL}).$$
\end{rem}
\begin{theorem}\label{thm-main}
Let $F:\Sigma \rightarrow \mathbb{C}P^2$ be a closed orientable Lagrangian surfaces of genus $g\geq 1$, then 
$$\mathcal{W}^-(F)\geq \frac{8\pi^2}{3\sqrt{3}}=\mathcal W^-(T_{\CL}),\qquad\quad\mathcal{W}(F)\geq \frac{4\pi^2}{3\sqrt{3}}=\mathcal W(T_{\CL})
$$
with equality attained if and only if $F$ is congruent to the Clifford torus in $\mathbb{C}P^2$. 
\end{theorem}
\begin{proof}It suffices to prove the inequality for $g \ge 2$. In this case, $F$ cannot be an embedding due to a topological obstruction in the normal bundle. In particular, the Lagrangian condition implies that its Euler number is non-zero, which implies that the self-intersection number does not vanish; See \cite[Lemma~4.6]{Eschenburg1985} for details. Therefore, the maximum multiplicity $\mu(F)\geq 2$. Then it follows from Theroem~\ref{thm-MU} that 
$$ \mathcal{W}^{-}(F)\geq4\pi\mu(F)\geq 8\pi> \frac{8\pi^2}{3\sqrt{3}} .~~~~~~ \mathcal{W}(F)\geq\frac{\mathcal{W}^-(F)}{2}
> \frac{4\pi^2}{3\sqrt{3}}. 
$$
This completes the proof. 
\end{proof}
\begin{rem}
   As pointed out by Marques and Neves in \cite{Marques-Neves2}, this lower bound for the functional $\mathcal{W}^-$ has an immediate consequence for the singularity theory of calibrated submanifolds. Via the Hopf fibration $\pi: \mathbb{S}^5 \to \mathbb{C}P^2$, the horizontal lift of a minimal Lagrangian surface yields a minimal Legendrian surface in $\mathbb{S}^5$. Taking the cone over this two-dimensional Legendrian link produces a three-dimensional special Lagrangian cone in $\mathbb{C}^3$. Because the functional $\mathcal{W}^-$ is intimately related to the area of this link in $\mathbb{S}^5$, it directly governs the density of the cone. Consequently, the sharp inequality in Theorem~\ref{thm-main} establishes that, among all special Lagrangian cones in $\mathbb{C}^3$ whose link is a $3$-fold covering over its Hopf projection image, the cone generated by the Clifford torus is the unique non-trivial one  attaining the minimal possible density at the origin. We refer to \cite{Haskins1,Haskins2} for more related discussions. 
   We also note that Luo, Ma and Yin in \cite{LMY} characterized the Clifford torus as the unique embedded minimal Lagrangian torus in $\mathbb{C}P^2$.
\end{rem}
\section{A disproof of Montiel-Urbano's conjecture in the non-Lagrangian setting}
\label{sec:counterexamples}
In this section, we show that there exist tori in $\mathbb{C}P^2$ whose $\mathcal{W}^-$ energy is less than that of the Clifford torus $T_{\CL}$. Hence, Montiel-Urbano's conjecture fails for non-Lagrangian tori in $\mathbb{C}P^2$.
Such examples are obtained by deforming $T_{\CL}$ along non-Killing Jacobi fields of $\mathcal{W}^-$. 

\subsection{The $\mathcal{W}^-$-stability of the Clifford torus among all tori in $\mathbb{C}P^2$}
In Remark~6.4 of \cite{WangXie}, C. P. Wang and the second author stated the second variational formula of $\mathcal{W}^-$ without proof, and asserted that $T_{\CL}$ is stable with respect to $\mathcal{W}^-$. To better understand the Killing fields involved, we provide a direct proof here.
 
\begin{theorem}
\label{thm:quadratic-stability}
The Clifford torus $T_{\CL}$ is stable for $\mathcal{W}^-$ among all smooth tori in $\CP^2$, i.e., for any given smooth normal vector field $V$, we have  
$$D^2\mathcal{W}^-_{T_{\rm Cl}}[V,V]\geq0.$$ 
Moreover, its nullity is $12$, and the space of non-Killing Jacobi fields has dimension $6$. 
\end{theorem}

\begin{proof}
Let $L$ denote the Jacobi operator for the area functional. Given a smooth normal variation vector field $V$, it follows from $H = 0$ that the second variation of $\mathcal{W}^-$ evaluated at $T_{\CL}$ is given by
\begin{equation}\label{eq-2nd}\begin{split} D^2\mathcal{W}^-_{T_{\CL}}[V,V] &= 2 D^2 \mathcal{A}_{T_{\CL}}(V,V) + 2 \int_{T_{\CL}} |\delta H|^2 \dd A\\
&= 2 \int_{T_{\CL}} \langle V, LV \rangle \dd A + \frac{1}{2} \int_{T_{\CL}} \langle LV, LV \rangle \dd A. \\
\end{split}
\end{equation}
where we have used the standard fact that $\delta H = -\frac{1}{2}LV$. 

On the Clifford torus $T_{\CL}$, it is well known that 
$L = \Delta^\perp - 6$, where $\Delta^\perp$ is the rough Laplacian in the normal bundle. 
Since $T_{\CL}$ is a flat Lagrangian torus, 
its normal bundle $NT_{\CL}$ is also flat. Moreover, there exists a global parallel orthonormal frame $\{n_1, n_2\}$ for $NT_{\CL}$, satisfying $\nabla^\perp n_1 = \nabla^\perp n_2 = 0$. Using this frame, any normal vector field $V$ can be globally decomposed into scalar functions as $V = u n_1 + v n_2$, where $u, v \in C^\infty(T_{\CL})$. Then we have, 
$$LV = (\Delta u - 6u)n_1 + (\Delta v - 6v)n_2,$$
where  $\Delta$ is the Laplace-Beltrami operator on $T_{\CL}$. Substituting this into \eqref{eq-2nd} yields 
\begin{equation}
\begin{aligned}
D^2\mathcal{W}^-_{T_{\CL}}[V,V] = \frac{1}{2} \int_{T_{\CL}}  (\Delta u - 2u)(\Delta u- 6u) \,\dd A+\frac{1}{2} \int_{T_{\CL}}  (\Delta v- 2v)(\Delta v - 6v) \, \dd A.
\end{aligned}
\end{equation}
Therefore, the functional $\mathcal W^-$ is stable if and only if the Laplace spectrum of $T_{\CL}$ has no intersection with $(2, 6)$. 

The spectrum of the Laplacian of the Clifford torus $T_{\mathrm{\CL}}$ in $\mathbb{C}P^2$ is explicitly known. Here we refer to proofs of \cite[Theorem~6.3]{WangXie}. In particular, the first nonzero eigenvalue is exactly $\lambda_1 = 6$, with multiplicity $6$. We conclude that $D^2\mathcal{W}^-_{T_{\mathrm{\CL}}}[V,V] \geq 0$ for all $V$, and  $\dim ker(D^2\mathcal{W}^-_{T_{\mathrm{\CL}}})=12$. Finally, note that the dimension of isometric (and also conformal) transformations of $\mathbb{C}P^2$ that do not preserve the Clifford torus is exactly $6$. Hence, the space generated by non-Killing Jacobi fields has dimension $6$.
\end{proof}

\subsection{A $3$-parameter family of deformations of $T_{\CL}$ decreasing $W^-$}
To construct the required deformation using non-Killing Jacobi fields, we parameterize the Clifford torus $T_{\CL}$ as 
$$
[z_1:z_2:z_3]
=
[\frac 1{\sqrt{3}}e^{ix}:\frac 1{\sqrt{3}}e^{iy}:\frac 1{\sqrt{3}}],~~(x,y)\in (\R/2\pi\mathbb Z)^2.
$$ 
In terms of this parameterization, it is easy to verify that the Laplacian is
\[
\Delta
=
-6(\partial_x^2+\partial_x\partial_y+\partial_y^2).
\]

We consider the following three eigenfunctions of $T_{\CL}$ corresponding to the first eigenfunction $\lambda_1=6$, 
$$
\phi_1=\cos x,
\quad
\phi_2=\cos y,
\quad
\phi_3=\cos(x-y).
$$
For parameters $\alpha,\beta,\gamma\in\mathbb R$, 
\begin{equation}
f_{\alpha,\beta,\gamma}
\triangleq
\alpha\cos x+\beta\cos y+\gamma\cos(x-y),\qquad
 g_{\alpha,\beta,\gamma}
\triangleq
-2\alpha\cos x
-\frac{\beta}{2}\cos y
+\gamma\cos(x-y).
\end{equation}
It follows that 
\begin{equation}(\Delta-6)f_{\alpha,\beta,\gamma}=0,\qquad (\Delta-6)g_{\alpha,\beta,\gamma}=0,\qquad
6f_x+3f_y+3g_x+6g_y=0.
   \label{eq:null-family}
\end{equation}
For sufficiently small $|t|$, define
\begin{equation}
F_t^{\alpha,\beta,\gamma}(x,y)
\triangleq 
[
\sqrt{p_1}e^{ix}:
\sqrt{p_2}e^{iy}:
\sqrt{p_3}
],~~~(x,y)\in (\R/2\pi\mathbb Z)^2,
\label{eq:general-family}
\end{equation}
with 
$$
p_1=\frac13+t f_{\alpha,\beta,\gamma},
\quad
p_2=\frac13+t g_{\alpha,\beta,\gamma},
\quad
p_3=\frac13-t(f_{\alpha,\beta,\gamma}
                 +g_{\alpha,\beta,\gamma}).
$$

\begin{theorem} 
\label{thm:cubic-instability}
For every fixed $(\alpha,\beta,\gamma)\in\mathbb R^3$, the torus defined in the family
\eqref{eq:general-family} satisfies
\begin{align}
\mathcal{A}(F_t^{\alpha,\beta,\gamma})
&=
\left(
1+
\frac{81}{4}\alpha\beta\gamma\,t^3
\right)\mathcal{A}(T_{\CL})
+
O(t^4),
\label{eq:area-general-cubic}\\
\mathcal{W}^-(F_t^{\alpha,\beta,\gamma})
&=
\left(
1+
\frac{81}{4}\alpha\beta\gamma\,t^3
\right)\mathcal{W}^-(T_{\CL})
+
O(t^4).
\label{eq:Wminus-general-cubic}
\end{align}
Hence the Clifford torus is not a local minimizer of $\mathcal{W}^-$ among all smooth tori in $\mathbb{C}P^2$.
\end{theorem}
\begin{proof}
In the open toric chart parameterized by  
$[\sqrt{p_1}e^{ix}:\sqrt{p_2}e^{iy}:\sqrt{p_3}]$ with $p_3=1-p_1-p_2>0$, the Fubini-Study metric of $\mathbb{C}P^2$ takes the following form, 
\[
 g_{\mathrm{FS}}=\sum_{j=1}^3\frac{\mathrm dp_j^2}{4p_j}
       +p_1\mathrm dx^2+p_2\mathrm dy^2
       -(p_1\mathrm dx+p_2\mathrm dy)^2.
\]
Write $q=(x,y)^T$ and $p=(p_1,p_2)^T$. Then 
\begin{equation}\label{eq:metric}
 g_{\mathrm{FS}}=\frac14\mathrm dp^TG(p)^{-1}\mathrm dp
                       +\mathrm dq^TG(p)\mathrm dq,
 \qquad
 G(p)=\begin{pmatrix}p_1-p_1^2&-p_1p_2\\-p_1p_2&p_2-p_2^2\end{pmatrix}.
\end{equation}

For simplicity, we write $f=f_{\alpha,\beta,\gamma}$ and $g=g_{\alpha,\beta,\gamma}$. Set
\[v\triangleq (f,g)^T\quad \hbox{ and
}~~M\triangleq Dv=\left(\begin{matrix}f_x&f_y\\g_x&g_y\end{matrix}\right).\] Then 
we have $\mathrm dp=tM\mathrm dq$.
Therefore, the induced metric on $F_t^{\alpha,\beta,\gamma}$ is exactly $$h_t = G(p(t)) + \frac{t^2}{4} M^T G(p(t))^{-1} M.$$

Let $U=\max(\|f\|_\infty,\|g\|_\infty,\|f+g\|_\infty)$.
If $U=0$, the family is constant and all the asserted expansions are
immediate. Otherwise, for $|t|\le 1/(6U)$ we have $p_j\ge1/6$. 
Each $F_t$ is a graph over $(\R/2\pi\mathbb Z)^2$. 
It is injective and its differential has rank two, hence it is an
embedding by compactness. 

According to the parameter variations $p_1 = 1/3 + tf$ and $p_2 = 1/3 + tg$, we define the matrix $L$ as
\begin{equation}
    L \triangleq \frac{1}{3} \begin{pmatrix} f & -(f+g) \\ -(f+g) & g \end{pmatrix}.
\end{equation}
Then
\[
G(p(t))
=
G_0+tL-t^2vv^T,
\quad
G(p(t))^{-1}
=
B_0-tB_0LB_0+O(t^2),
\]
where \[
G_0
=
\frac19
\begin{pmatrix}
2&-1\\
-1&2
\end{pmatrix}, \qquad
B_0=G_0^{-1}
=
\begin{pmatrix}
6&3\\
3&6
\end{pmatrix}.
\]
Thus
\[
h_t
=
G_0+th_1+t^2h_2+t^3h_3+O(t^4),
\]
with 
$$
h_1=L,~~~ h_2=-vv^T+\frac14M^TB_0M,~~~h_3=-\frac14M^TB_0LB_0M.$$ 
To calculate the Taylor expansion of $\sqrt{\det(h_t)}$, we use the following claim. 

{\bf Claim.} {\em For arbitrary $2\times2$ matrices
$C_1,C_2,C_3$ with $\tr C_1=0$, there holds}
\begin{equation}\label{eq-determinant}\sqrt{\det(\Id+tC_1+t^2C_2+t^3C_3+O(t^4))}\notag=1+\frac{t^2}{2}(\tr C_2+\det C_1)
       +\frac{t^3}{2}\bigl(\tr C_3-\tr(C_1C_2)\bigr)+O(t^4).
\end{equation}
This claim follows from 
$$\det(\Id+tC_1+t^2C_2+t^3C_3+O(t^4)) = 1 + t^2(\text{tr} C_2 + \det C_1) + t^3(\text{tr} C_3 - \text{tr}(C_1C_2)) + O(t^4),$$
where the following identities for $2\times2$ matrices are used,   
\[\begin{split}
   \det(\Id+X)=&1+\tr X+\det X, \\
   \det(A+B)=&\det A + \det B + \text{tr}A\,\text{tr}B - \text{tr}(AB).
\end{split}\]
Notice that $\text{tr}(B_0L)=0$. Applying the above claim to $C_j=B_0h_j$, we obtain
\begin{equation}\label{eq:density}
 \mathrm dA_t=(1+t^2a_2+t^3a_3+O(t^4))\mathrm dA_0,
\end{equation}
where, by writing 
$N=B_0LB_0$,
\begin{align}
 a_2&=-\frac34v^TB_0v+\frac18\tr(B_0M^TB_0M),\label{eq:a2}\\
 a_3&=\frac12v^TNv-\frac18\bigl\{
       \tr(B_0M^TNM)+\tr(NM^TB_0M)\bigr\}.\label{eq:a3}
\end{align}
Here the quadratic calculation uses
$v^TB_0v=6(f^2+fg+g^2)$ and $\det(B_0L)=-3(f^2+fg+g^2)$; the cubic calculation follows by
substituting $h_2,h_3$ into the coefficient in
\eqref{eq-determinant}. Note that we also have
\begin{equation}
    \label{eq-a2}
    \int_{T^2} a_2\,\dd A=-\frac{3}{4}\int_{T^2} v^T B_0 v\,\dd A+\frac{1}{8}\int_{T^2}v^T B_0 \Delta v \,\dd A=\frac{1}{8}\int_{T^2}v^T B_0 (\Delta v-6v)\,\dd A=0. 
\end{equation}
It follows from \eqref{eq:null-family} that 
\begin{equation}\label{eq:traceM}
 \tr(B_0M)=6f_x+3f_y+3g_x+6g_y=0.
\end{equation}
Set
\[
 R=B_0^{1/2}MB_0^{1/2},\qquad
 S=B_0^{1/2}LB_0^{1/2}.
\]
We have $\tr R=0$ by \eqref{eq:traceM} and $\tr S=0$ by
$\tr(B_0L)=0$. Observe that every real trace-free $2\times2$ matrix satisfies
\begin{equation}\label{eq:matrix-identity}
 RR^T+R^TR=\tr(R^TR)\Id.
\end{equation}
Consequently, cyclic
invariance of trace gives
\begin{align*}
&\tr(B_0M^TNM)+\tr(NM^TB_0M)=\tr(R^TSR)+\tr(SR^TR)
        =\tr\bigl(S(RR^T+R^TR)\bigr)=0.
\end{align*}
Finally, direct multiplication gives
\[
 N=-9\begin{pmatrix}g&f+g\\f+g&f\end{pmatrix},\qquad
 v^TNv=-27fg(f+g).
\]
Thus \eqref{eq:a3} reduces pointwise to
\begin{equation}\label{eq:a3-simple}
 a_3=-\frac{27}{2}fg(f+g).
\end{equation}
Hence, 
\begin{equation}
\frac1{(2\pi)^2}
\int_0^{2\pi}\int_0^{2\pi}a_3(x,y)\,dx\,dy
=
\frac{81}{4}\alpha\beta\gamma.
\label{eq:a3-average-general}
\end{equation}

Since \eqref{eq:null-family} implies that the full area Hessian
vanishes against every normal field,
\[
D^2A_0[v,w]=0
\qquad
\text{for all }w,
\]
the first variation of the mean curvature also vanishes: 
$\left.\nabla_tH_{F_t}\right|_{t=0}=0.
$
Therefore
$
H_{F_t}=O(t^2)
$
uniformly and
\begin{equation}
E_H(F_t)
=
\int |H_{F_t}|^2\,dA_t
=
O(t^4).
\label{eq:EH-order4-general}
\end{equation}
Equations \eqref{eq-a2}, \eqref{eq:a3-average-general} and
\eqref{eq:EH-order4-general} give
\eqref{eq:area-general-cubic} and
\eqref{eq:Wminus-general-cubic}. 
\end{proof}

\section{Two open problems}\label{sec-5}
In the last section, we showed that the Clifford torus fails to minimize $\mathcal{W}^-$ among all smooth tori. This motivates the following problem.   
\begin{problem}
\label{prob:Wminus-minimizer}  For the Willmore-type functional $\mathcal{W}^-$, define
\[
\beta^-
=
\inf\left\{
\mathcal{W}^-(F):F:T^2\to\mathbb{C}P^2\text{ is an embedding}
\right\}.
\]
Determine the value of $\beta^-$ and whether it is attained by a smooth embedding. If so, describe the geometry of minimizer. In particular, determine whether such a torus lies in a bifurcation branch generated from the non-Killing Jacobi fields of the Clifford torus.
\end{problem}
We refer to \cite{BK,Kuwert,KS, Marques-Neves, Minicozzi,Riviere,Simon} for some related works which could be useful for this problem.\vspace{3mm}

In \cite{WangXie}, C. P. Wang and the second author established the strict stability of the Clifford torus for the whole Willmore functional $\mathcal{W}$ along all smooth variations. It is natural to adapt Montiel-Urbano's conjecture for $\mathcal{W}^-$ among all smooth tori to the Willmore  functional $\mathcal{W}$.
\begin{conjecture}
\label{conj:ordinary-W-global}
The Clifford torus $T_{\CL}$ achieves the minimum of the Willmore functional $\mathcal{W}$ amongst all smooth tori $\mathbb{C}P^2$. 
\end{conjecture}

\vskip 0.5cm
\textbf{Acknowledgement}
This work is supported by NSFC Nos. 12671061, 12371052 and 12171473. The second author is also partially supported by the Fundamental Research Funds for Central Universities. 
\vskip 0.3cm
\textbf{AI Disclosure.} ChatGPT assistance was used extensively in this work. We first used AI to verify the conjecture for special families of tori, and then guided it to find the right way to approach the conjecture in the Lagrangian setting, and search for counterexamples in the non-Lagrangian setting. The authors directed the investigation, examined gaps in the arguments, and decided which approaches to retain or revise. The final proof was substantially rewritten and reinterpreted by the authors. Two open problems are also proposed by the authors. AI tools were also used for language polishing. The authors take full responsibility for the mathematical content and conclusions. 

\end{document}